\documentclass[preprint,12pt]{elsarticle}
\usepackage{}
\usepackage{amsfonts}
\usepackage{mathrsfs}
\usepackage{graphicx,amsmath,amsfonts,amssymb,amscd,amsthm, mathrsfs}
\newtheoremstyle{kai}
{3pt}{3pt}{}{}{\bfseries}{.}{.5em}{}
\makeatletter
\def\EquationsBySection{\def\theequation
{\thesection.\arabic{equation}}%
\@addtoreset{equation}{section}}

\newcommand\old[1]{}
\newcommand{\pend}{\hfill \thicklines \framebox(6.6,6.6)[l]{}}
\renewenvironment{proof}{\noindent {\it  Proof.} \rm}{\pend}
\newtheorem{theorem}{Theorem}[section]
\newtheorem{lemma}{Lemma}[section]

\newtheorem{proposition}{Proposition}[section]

\newtheorem{definition}{Definition}[section]
\newtheorem{example}{Example}[section]

\journal{arXiv.org}
\EquationsBySection
\makeatother

\begin{document}

\begin{frontmatter}

\title{\bf\Large Stationary Solutions of Neutral Stochastic Partial\\ Differential Equations with Delays in \\ the Highest-Order Derivatives}

%\author{{\bf Kai Liu}}
%\address{
%Department of Mathematical Sciences, School of Physical Sciences,
%The University of Liverpool, Liverpool, L69 7ZL, U.K.}

\author{
{\bf Kai Liu}}
\address{
$^{a)}$ College of Mathematical Sciences, Tianjin Normal University, Tianjin, 300387, P.R. China, and $^{b)}$ Department of Mathematical Sciences, School of Physical Sciences,
The University of Liverpool, Liverpool, L69 7ZL, U.K. E-mail: k.liu@liverpool.ac.uk}

 \tnotetext[tref1]{The author is grateful to the Tianjin Thousand Talents Plan for its financial support.}

\begin{abstract}
%% Text of abstract

In this work, we shall consider the existence and uniqueness of stationary solutions to stochastic partial functional differential equations with additive noise in which a neutral type of delay is explicitly presented. We are especially concerned about those delays appearing in both spatial and temporal derivative terms in which the coefficient operator  under spatial variables may take the same form as the  infinitesimal generator of the equation. We establish the stationary property of the neutral system under investigation  by focusing on distributed delays. In the end, an illustrative example is analyzed to explain  the theory in this work.\\ \\
\end{abstract}
\end{frontmatter}

\noindent {\bf Keywords:} Stochastic functional differential equation of neutral type; Strongly continuous or $C_0$ semigroup; Resolvent operator; Stationary solution.

\noindent {\bf 2010 Mathematics Subject Classification(s):} 60H15, 60G15, 60H05.

% \linenumbers

\newpage
\section{\large Introduction}

First of all, let us consider some simple stochastic systems to motivate our theory in this work. Let $w(t)$, $t\ge 0$, be a standard real Brownian motion defined on some probability space $(\Omega, {\mathscr F}, {\mathbb P})$. Consider the following stochastic partial differential equation
\begin{equation}
\label{11/08/2013(3007089)}
\begin{cases}
dy(t, \xi) =\displaystyle\frac{\partial^2}{\partial \xi^2}y(t, \xi)dt + b(\xi) dw(t),\,\,\,\,t\ge 0,\,\,\,\,\xi\in (0, \pi),\\
y(t, 0)=y(t, \pi)=0,\,\,\,t\ge 0,\\
y(0, \cdot)=y_0(\cdot)\in L^2(0, \pi),
\end{cases}
\end{equation}
where $b(\cdot)\in L^2(0, \pi)$. It is well-known  (see, e.g., Pr\'ev\^ot and R\"ockner \cite{cpmr2007}) that equation (\ref{11/08/2013(3007089)}) has a unique stationary solution. That is, there exists a random initial $y_0\in L^2(0, \pi)$ such that the corresponding (strong) solution $y(t, y_0)$, $t\ge 0$, is stationary, i.e., for any $t\ge 0$, $t_k\ge 0$ and Borel set $\Gamma_k\in {\mathscr B}(H)$, the Borel $\sigma$-field on $H$, $k=1,\ldots, n$,
\[
{\mathbb P}\big\{y(t+t_k, y_0)\in \Gamma_k,\,\,k=1,\ldots, n\big\}= {\mathbb P}\big\{y(t_k, y_0)\in \Gamma_k,\,\,k=1,\ldots, n\big\}.\]
Moreover, this stationary solution is unique in the sense that any two stationary solutions of (\ref{11/08/2013(3007089)}) have the same finite dimensional distribution.

 Next, let $r>0$ and consider a time delay version of (\ref{11/08/2013(3007089)}) in the form
\begin{equation}
\label{05/03/2017}
\begin{cases}
 dy(t, \xi)=\displaystyle\frac{\partial^2}{\partial \xi^2}y(t, \xi)dt  + \displaystyle\int^0_{-r} \beta(\theta) \frac{\partial^2}{\partial \xi^2}y(t+\theta, \xi)d\theta dt + b(\xi)  dw(t),\,\,\,\,t\ge 0,\,\,\,\,\,\xi\in (0, \pi),\\
y(t, 0)=y(t, \pi)=0,\,\,\,t\ge 0,\\
y(0, \xi)=\phi_0(\xi),\,\,\,\,y(\theta, \xi)=\phi_1(\theta, \xi),\,\,\,\theta \in [-r, 0],\,\,\,\,\xi\in (0, \pi),
\end{cases}
\end{equation}
where $\beta:\, [-r, 0]\to {\mathbb R}$ is some measurable function. The novelty in equation (\ref{05/03/2017}) is that a time delay appears in the highest-order, i.e., second order,  derivative term which usually leads to an unbounded operator from an advanced analysis viewpoint. On the other hand, due to the time delay in (\ref{05/03/2017}), it is essential to set up proper initial data, e.g., $\phi_0\in L^2(0, \pi)$ and $\phi_1\in L^2([-r, 0], H^1_0(0, \pi))$ where $H^1_0(0, \pi)$ is the classical Sobolev space, to secure a solution, and further a stationary solution, to equation (\ref{05/03/2017}). As a matter of fact, it has been shown in Liu \cite{kl2016(1)} that if
\[
\|\beta\|_{L^1([-r, 0], {\mathbb R})}<1,\]
 there would exist a unique stationary solution to (\ref{05/03/2017}).

In this work, we are interested in a neutral type of version of (\ref{05/03/2017}) in the form
\begin{equation}
\label{11/08/2013(30070999)}
\begin{cases}
 d\Big(y(t, \xi) - \displaystyle\int^0_{-r} \gamma(\theta)y(t+\theta, \xi)d\theta\Big)=\displaystyle\frac{\partial^2}{\partial \xi^2}y(t, \xi)dt  + \displaystyle\int^0_{-r} \beta(\theta) \frac{\partial^2}{\partial \xi^2}y(t+\theta, \xi)d\theta dt\\
  \hskip 200pt +\, b(\xi)  dw(t),\,\,\,\,t\ge 0,\,\,\,\,\xi\in (0, \pi),\\
y(t, 0)=y(t, \pi)=0,\,\,\,t\ge 0,\\
y(0, \xi)=\phi_0(\xi),\,\,\,\,y(\theta, \xi)=\phi_1(\theta, \xi),\,\,\,\theta \in [-r, 0],\,\,\,\xi\in (0, \pi),
\end{cases}
\end{equation}
or equivalently, the form
 \begin{equation}
\label{11/08/2013(3007099989)}
\begin{cases}
 d\Big(y(t, \xi) - \displaystyle\int^0_{-r} \gamma(\theta)y(t+\theta, \xi)d\theta\Big)=\displaystyle\frac{\partial^2}{\partial \xi^2}\Big(y(t, \xi) - \displaystyle\int^0_{-r} \gamma(\theta)y(t+\theta, \xi)d\theta\Big) dt\\
   \hskip 130pt + \displaystyle\int^0_{-r} (\beta(\theta)+\gamma(\theta)) \frac{\partial^2}{\partial \xi^2}y(t+\theta, \xi)d\theta dt + b(\xi)  dw(t),\,\,\,\,t\ge 0,\\
y(t, 0)=y(t, \pi)=0,\,\,\,t\ge 0,\\
y(0, \xi)=\phi_0(\xi),\,\,\,\,y(\theta, \xi)=\phi_1(\theta, \xi),\,\,\,\theta \in [-r, 0],\,\,\,\xi\in(0, \pi),
\end{cases}
\end{equation}
where $\gamma:\, [-r, 0]\to {\mathbb R}$ is a measurable function. As a direct result of our theory, we shall show later on that if
\[
2\|\gamma\|_{L^1([-r, 0], {\mathbb R})} + \|\beta\|_{L^1([-r, 0], {\mathbb R})}<1,\]
there would exist a unique stationary solution to (\ref{11/08/2013(30070999)}).

The organization of this  work is as follows. In Section 2, we first develop a $C_0$-semigroup theory so as to lift up the original time delay system into a non time delay one.  To identify a stationary solution for our system, it is important to know when the associated ``lift-up" solution semigroup  is exponentially stable. To this end, we establish some stability results by means of a spectrum analysis method in Section 3. In contrast with point delay situation, it turns out in Section 4 that we can apply a norm continuity result of $C_0$-semigroups in \cite{kl2016(2)}
to our case to locate a stationary solution for the system under consideration. Last, we shall apply the results established in the work to a concrete  example to illustrate our theory.

\section{\large Strongly Continuous Semigroup}

For arbitrary Banach spaces $X$ and $Y$ with their respective norms $\|\cdot\|_X$ and $\|\cdot\|_Y$, we always denote by ${\mathscr L}(X, Y)$ the space of all bounded, linear operators from $X$ into $Y$. If $X=Y$, we simply write ${\mathscr L}(X)$ for ${\mathscr L}(X, X)$.  Let $V$ be a separable Hilbert space and  $a:\, V\times V\to {\mathbb R}$  a  bilinear form satisfying the so-called
G\r{a}rding's inequalities
\begin{equation}
\label{29/01/2013(390)}
|a(x, y)|\le \beta\|x\|_V\|y\|_V,\hskip 15pt a(x, x)\le -\alpha\|x\|^2_V,\hskip 20pt \forall\, x,\,y\in V,
\end{equation}
for some constants $\beta>0$, $\alpha>0$.  In association with the form $a(\cdot, \cdot)$, let $A$ be a linear operator defined by
\[
a(x, y)= \langle x, Ay\rangle_{V, V^*},\hskip 15pt x,\,\,y\in V,
\]
where $V^*$ is the dual space of $V$ and $\langle \cdot, \cdot\rangle_{V, V^*}$ is the dual pairing between $V$ and $V^*$.
Then $A\in {\mathscr L}(V, V^*)$ and $A$ generates an analytic semigroup $e^{tA}$, $t\ge 0$, on $V^*$. We also introduce the standard interpolation  Hilbert space $H = (V, V^*)_{1/2, 2}$, which is  described by
\[
H = \Big\{x\in V^*: \int^\infty_0 \|Ae^{tA}x\|^2_{V^*}dt<\infty\Big\}\]
with inner product
\[
\langle x, y\rangle_H = \langle x, y\rangle_{V^*}  + \int^\infty_0 \langle Ae^{tA}x, Ae^{tA}y\rangle_{V^*}dt,\hskip 15pt x,\,\,y\in V^*.\]
We identity the dual $H^*$ of $H$ with $H$, then it is easy to see that
\begin{equation}
\label{13/04/16(100)}
V\hookrightarrow H = H^*\hookrightarrow V^*
\end{equation}
where the imbedding $\hookrightarrow$ is dense and continuous with  $\|x\|^2_{H}\le \nu\|x\|^2_{V}$, $x\in V,$ for some constant $\nu >0$.
Hence, $\langle x, Ay\rangle_H = \langle x, Ay\rangle_{V, V^*}$ for all $x\in V$ and $y\in V$ with $Ay\in H$. Moreover,  for any $T\ge 0$ it is well known   (see \cite{jllem72}) that
\[
L^2([0, T], V)\cap W^{1, 2}([0, T], V^*)\subset C([0, T], H)\]
where $W^{1, 2}([0, T], V^*)$ is the Sobolev space consisting of all functions $y:\, [0, T]\to V^*$ such that $y$ and its first order distributional derivative are in $L^2([0, T], V^*)$ and $C([0, T], H)$ is the space of all continuous functions from $[0, T]$ into $H$, respectively.
It can be also shown (see, e.g., \cite{nt79}) that the semigroup $e^{tA}$, $t\ge 0$, is  bounded and analytic on both $V^*$ and $H$ such that $e^{tA}:\, V^*\to V$ for each $t>0$ and  for some constant $M>0$,
\[
\|e^{tA}\|_{{\mathscr L}(V^*)}\le M,\hskip 20pt
\|e^{tA}\|_{{\mathscr L}(H)}\le e^{-\alpha t}\hskip 15pt \hbox{ for all}\hskip 15pt t\ge 0.\]

Let $T\ge 0$ and  $f\in L^2([0, T], V^*)$. Consider an abstract evolution equation in $V^*$ as follows:
\begin{equation}
\label{25/04/16(20)}
\begin{cases}
dx(t) = Ax(t)dt +f(t)dt,\hskip 20pt t\in [0, T],\\
x(0)=\phi_0.
\end{cases}
\end{equation}
The proofs of the following result are referred to Theorems 2.3 and 2.4 in  \cite{dbgkkes84}.

\begin{theorem}
\label{23/05/16(10)}
\begin{enumerate}
\item[(i)] If function $f\in L^2([0, T], V^*)$, then
\[
(e^A*f)(t) := \int^t_0 e^{(t-s)A}f(s)ds\in L^2([0, T], V)\cap W^{1, 2}([0, T], V^*),\]
and
\[
\|e^A*f\|_{L^2([0, T], V)}\vee \|e^A*f\|_{W^{1, 2}([0, T], V^*)}\le C_1\|f\|_{L^2([0, T], V^*)}\]
where $C_1=C_1(T)>0$ and $a\vee b := \max\{a, b\}$ for any $a,\,b\in {\mathbb R}$.
\item[(ii)] If $\phi_0\in H$, then function $t\to e^{tA}\phi_0$ belongs to $L^2([0, T], V)\cap W^{1, 2}([0, T], V^*)$ and
\[
\|e^{tA}\phi_0\|_{L^2([0, T], V)}\le \|e^{tA}\phi_0\|_{W^{1, 2}([0, T], V^*)}\le C_2\|\phi_0\|_H\]
where $C_2=\max\{M, \sqrt{T}, 1\}$.
\item[(iii)] If $f\in L^2([0, T], V^*)$ and $\phi_0\in H$, then equation (\ref{25/04/16(20)}) has a unique solution given by
\[
x(t) = e^{tA}\phi_0 + \int^t_0 e^{(t-s)A}f(s)ds.\]
In addition, there exists a constant $C_0>0$ such that $x\in C([0, T], H)$ and
\[
\begin{split}
C_0\|x\|_{C([0, T], H)}\le \|x\|_{L^2([0, T], V)}\vee \|x\|_{W^{1, 2}([0, T], V^*)}\le C_1\|f\|_{L^2([0, T], V^*)}+ C_2\|\phi_0\|_H.
\end{split}
\]
\item[(iv)] If $f\in W^{1, 2}([0, T], V^*)$, $\phi_0\in V$ and $A\phi_0 + f(0)\in H$,  then the solution $x$ of
(\ref{25/04/16(20)}) satisfies
\[
x\in W^{1, 2}([0, T], V)\cap W^{2, 2}([0, T], V^*)\subset C^1([0, T], H),\]
\[
\|x'\|_{L^2([0, T], V)}\vee \|x'\|_{W^{1, 2}([0, T], V^*)}\le C_1\|f'\|_{L^2([0, T], V^*)} + C_2\|A\phi_0 + f(0)\|_H,\]
and
\[
x'(t) = e^{tA}(A\phi_0 + f(0)) + \int^t_0 e^{(t-s)A}f'(s)ds,\hskip 15pt t\in [0, T].\]
\item[(v)]
If $f\in L^2([0, T], V)$, $\phi_0\in V$ and $A\phi_0\in H$, then the solution $x$ of  (\ref{25/04/16(20)}) satisfies
\[
Ax\in L^2([0, T], V)\cap W^{1, 2}([0, T], V^*)\subset C([0, T], H),\]
\[
\|Ax\|_{L^2([0, T], V)} + \|Ax\|_{W^{1, 2}([0, T], V^*)}\le C_1\|Af\|_{L^2([0, T], V^*)} + C_2\|A\phi_0\|_H\]
and  for any $t\ge 0$,
\[
Ax(t) = e^{tA}A\phi_0 + \int^t_0 e^{(t-s)A}Af(s)ds.\]
\end{enumerate}
\end{theorem}

Let $r>0$ and $T\ge  0$. For $x\in L^2([-r, T], V)$, we always write  $x_t(\theta) := x(t+\theta)$ for any $t\ge 0$ and $\theta\in [-r, 0]$ in this work. Now  suppose that  $D_1\in {\mathscr L}(V)$, $D_2\in {\mathscr L}(L^2([-r, 0], V), V)$, $F_1\in {\mathscr L}(V, V^*)$ and $F_2\in {\mathscr L}(L^2([-r, 0], V), V^*).$  We introduce two linear mappings $D$ and ${F}$ on $C([-r, T], V)$, respectively, by
$$
\label{23/05/06(46789)}
Dx_t =D_1x(t-r) + D_2x_t,\hskip 20pt t\in [0, T],\hskip 20pt \forall\, x(\cdot)\in C([-r, T], V),
$$
and
$$
\label{23/05/06(4678)}
Fx_t =F_1x(t-r) + F_2x_t,\hskip 20pt t\in [0, T],\hskip 20pt \forall\, x(\cdot)\in C([-r, T], V).
$$

\begin{lemma}
\label{01/06/16(1)}
Both the mappings $D$ and $F$ have a bounded, linear extension to $L^2([-r, T], V)$ such that for any
 $x\in L^2([-r, T], V)$,
\begin{equation}
\label{26/04/16(1)}
\int^T_0 \|Dx_t\|^2_{V}dt \le C_1\int^T_{-r} \|x(t)\|^2_{V}dt
\end{equation}
 and
\begin{equation}
\label{26/04/16(2)}
\int^T_0 \|Fx_t\|^2_{V^*}dt \le C_2\int^T_{-r} \|x(t)\|^2_{V}dt
\end{equation}
with
\[
\begin{split}
C_1 &=  (\|D_1\|_{{\mathscr L}(V)} + \|D_2\|_{{\mathscr L}(L^2([-r, 0], V), V)}\cdot r^{1/2})^2>0,\\
C_2 &= ( \|F_1\|_{{\mathscr L}(V, V^*)} + \|F_2\|_{{\mathscr L}(L^2([-r, 0], V), V^*)}\cdot r^{1/2})^2>0.
\end{split}
\]
\end{lemma}
\begin{proof} We only prove (\ref{26/04/16(1)}) since the relation (\ref{26/04/16(2)}) can be obtained in an analogous manner.
By using H\"older's inequality and Fubini's theorem, we can obtain that for any  $x(\cdot)\in C([-r, T], V)$,
\[
\label{17/05/06(56758)}
\begin{split}
\Big(\int^{T}_0 \|Dx_t\|^2_{V} dt\Big)^{1/2}
&\le \Big(\int^{T}_0 \|D_1\|^2\|x(t-r)\|^2_{V} dt\Big)^{1/2} + \|D_2\| \Big(\int^{T}_0 \int^0_{-r} \|x(t+\theta)\|_{V}^2 d\theta dt\Big)^{1/2}\\
&\le  \|D_1\|\Big(\int^{T}_0 \|x(t-r)\|^2_{V} dt\Big)^{1/2} + \|D_2\|\Big(\int^0_{-r} \int^T_{-r} \|x(t)\|^2_{V} dt d\theta\Big)^{1/2}\\
&\le [ \|D_1\| + \|D_2\| r^{1/2}] \Big(\int^{T}_{-r} \|x(t)\|_{V}^2 dt\Big)^{1/2}.
\end{split}
\]
Since $C([-r, T], {V})$ is dense in $L^2([-r, T], {V})$, $D$ admits a bounded linear extension, still denote it by $D$, from $L^2([-r, T], {V})$ to $L^2([0, T], {V})$. The proof is thus complete.
\end{proof}

Let ${\cal H} = H\times L^2([-r, 0], V)$ and consider the following deterministic functional differential equation of neutral type in $V^*$,
\begin{equation}
\label{25/04/16(1)}
\begin{cases}
\displaystyle\frac{d}{dt}(x(t)-Dx_t) = A(x(t)-Dx_t) +Fx_t,\hskip 15pt t\ge 0,\\
x(0)=\phi_0 + D\phi_1,\,\,x_0=\phi_1,\,\,\phi=(\phi_0, \phi_1)\in {\cal H}.
\end{cases}
\end{equation}
The integral form of (\ref{25/04/16(1)}) is given  by
\begin{equation}
\begin{cases}
\label{25/04/16(2)}
x(t)-Dx_t = \displaystyle e^{tA}\phi_0 + \int^t_0 e^{(t-s)A}Fx_sds,\hskip 15pt t\ge 0,\\
x(0)=\phi_0 + D\phi_1,\,\,x_0=\phi_1,\,\,\phi=(\phi_0, \phi_1)\in {\cal H}.
\end{cases}
\end{equation}
We say that $x$ is a {\it (strict) solution\/} of  (\ref{25/04/16(1)}) in $[0, T]$ if $x\in L^2([0, T], V)\cap W^{1,2}([0, T], V^*)$ and the equation (\ref{25/04/16(1)}) is satisfied almost everywhere in $[0, T]$, $T\ge 0$.

\begin{theorem}
\label{29/04/16(1)}
Given arbitrarily $\phi=(\phi_0, \phi_1) \in H\times L^2([-r, 0], V)$ and $T\ge 0$, there exists a function $x(t)\in V$, $t\in [0, T]$, which is the unique solution  of equation (\ref{25/04/16(1)})   with initial $x(0)=\phi_0 + D\phi_1$ and $x_0=\phi_1$ such that
\[
 x(\cdot)\in L^2([0, T], V),\]
and
\[
 y(\cdot) := x(\cdot) - Dx_\cdot \in L^2([0, T], V)\cap W^{1, 2}([0, T], V^*)\subset C([0, T], H).\]
 Moreover, we have the relations
\begin{equation}
\label{27/04/16(1)}
 \|x\|_{L^2([0, T], V)} \le M\Big(\|\phi_0\|_H + \|\phi_1\|_{L^2([-r, 0], V)}\Big)
\end{equation}
and
\begin{equation}
\label{27/04/16(17890)}
 \|y\|_{L^2([0, T], V)} + \|y\|_{W^{1, 2}([0, T], V^*)}\le M\Big(\|\phi_0\|_H + \|\phi_1\|_{L^2([-r, 0], V)}\Big)
\end{equation}
for some positive number $M=M(T)>0$.
\end{theorem}
\begin{proof}
We shall use a fixed point argument to the integral equation (\ref{25/04/16(2)}). Let $t_0\in (0, r)$ and $\Sigma$ be a closed subspace of $L^2([-r, t_0], V)$ such that  $x(0)=\phi_0 + D\phi_1$ and $x_0=\phi_1$ for $x\in \Sigma$. Define a mapping $S$ on $\Sigma$ as follows:  for $x\in \Sigma$,
\begin{equation}
\label{23/05/16(60)}
\begin{cases}
(Sx)(t) = Dx_t + e^{tA}\phi_0+ \displaystyle\int^t_0 e^{(t-s)A} Fx_sds\hskip 10pt \hbox{for}\hskip 10pt t\in [0, t_0],\\
(Sx)(0)=\phi_0 + D\phi_1,\,\,\, (Sx)_0=\phi_1,\,\,\,\,(\phi_0, \phi_1)\in {\cal H}.
\end{cases}
\end{equation}
Note that we have
\begin{equation}
\label{02/06/16(11)}
\begin{split}
\|D_2x_\cdot\|_{L^2([0, t_0], V)}&\le \|D_2\|\Big(\int^{t_0}_0 \|x_t\|^2_{L^2([-r, 0], V)}dt\Big)^{1/2}\\
&= \|D_2\|\Big(\int^{t_0}_0 \int^0_{-r} \|x(t+\theta)\|^2_Vd\theta dt\Big)^{1/2}\\
&\le \|D_2\|\Big( \int^{t_0}_0 \int^{t_0}_{-r} \|x(s)\|^2_V ds dt\Big)^{1/2}=\sqrt{t_0}\|D_2\|\|x\|_{L^2([-r, t_0], V)}.
\end{split}
\end{equation}
Thus, by virtue of Theorem \ref{23/05/16(10)} (i) and (ii), it is immediate that   $Sx(\cdot)\in L^2([0, t_0], V)$ and $Sx(\cdot) - Dx_\cdot\in L^2([0, t_0], V)\cap W^{1, 2}([0, t_0], V^*)$ for each $x\in \Sigma$.

To obtain a unique solution of  (\ref{25/04/16(1)}), it suffices to show that $S$ is a contraction from $\Sigma$ into itself for sufficiently small $t_0>0$ and then implement a successive interval argument to extend the solution onto the whole interval $[0, T]$. Indeed, for any $x,\,\bar x\in \Sigma$ and $t\in [0, t_0]$, we have
\begin{equation}
\label{02/06/16(10)}
\begin{split}
(Sx -S\bar x)(t) = &\,\, D_1(x(t-r) -\bar x(t-r)) + D_2(x_t -\bar x_t)\\
&\,\, + \int^t_0 e^{(t-s)A} [F_1(x(s-r)-\bar x(s-r)) + F_2(x_s-\bar x_s)]ds\\
=&\,\, D_2(x_t -{\bar x}_t) + \int^t_0 e^{(t-s)A}F_2(x_s-{\bar x}_s)ds.
\end{split}
\end{equation}
which, in addition to (\ref{02/06/16(11)}) and (\ref{02/06/16(10)}),  immediately yields that
\begin{equation}
\label{23/05/16(15)}
\begin{split}
\|Sx-S\bar x&\|_{L^2([0, t_0], V)}\\
\le&\,\, \|D_2(x_\cdot-\bar x_\cdot)\|_{L^2([0, t_0], V)}+ \Big\|\int^\cdot_0 e^{(\cdot-s)A} F_2(x_s-\bar x_s)ds\Big\|_{L^2([0, t_0], V)}\\
\le&\,\, \|D_2\|\sqrt{t_0}\|x(\cdot)-\bar x(\cdot)\|_{L^2([-r, t_0], V)} + M_0 t_0 \|F_2\|\|x(\cdot)-\bar x(\cdot)\|_{L^2([-r, t_0], V)}\\
=&\,\, \delta(t_0)\|x(\cdot)-\bar x(\cdot)\|_{L^2([-r, t_0], V)},
\end{split}
\end{equation}
where $\delta(t_0)=  \|D_2\|\sqrt{t_0} +  M_0\|F_2\| t_0\to 0$ as $t_0\to 0$. The map $S$ is thus a contraction in $\Sigma$ and the equation (\ref{25/04/16(2)}) has a unique solution $x$ on $[-r, t_0]$.

To show the relation (\ref{27/04/16(1)}), we re-write (\ref{25/04/16(2)}) for $t\in [0, t_0]$, $t_0<r$, to get
\[
x(t) = Dx_t + e^{tA}\phi_0 + \int^t_0 e^{(t-s)A}Fx_sds.\]
Then, from Theorem \ref{23/05/16(10)} (i) and (ii) and (\ref{02/06/16(11)}), we obtain
\begin{equation}
\label{23/05/16(11)}
\begin{split}
\|x\|_{L^2([0, t_0], V)}\le&\,\, \sqrt{t_0}\|D_2\|\|x\|_{L^2([0, t_0], V)} +  C_2(t_0)\|\phi_0\|_H\\
&\,\, + C_1(t_0)\Big(\int^{t_0}_0 \|F_1\phi_1(t-r)\|^2_{V^*}dt\Big)^{1/2} + C_1(t_0)\Big(\int^{t_0}_0 \|F_2x_t\|^2_{V^*}dt\Big)^{1/2},
\end{split}
\end{equation}
where $C_1(t_0),\,C_2(t_0)>0$ are those numbers given in Theorem  \ref{23/05/16(10)}.
On the other hand, it is immediate that
\begin{equation}
\label{23/05/16(12)}
\Big(\int^{t_0}_0 \|F_1\phi_1(t-r)\|^2_{V^*}dt\Big)^{1/2} \le \|F_1\|\|\phi_1\|_{L^2([-r, 0], V)},
\end{equation}
and similarly to  (\ref{23/05/16(15)}), we have
\begin{equation}
\label{23/05/16(20)}
\Big(\int^{t_0}_0 \|F_2x_t\|^2_{V^*}dt\Big)^{1/2} \le \|F_2\|t_0 \Big(\int^{t_0}_0 \|x(t)\|^2_Vdt\Big)^{1/2}.
\end{equation}
Hence, by letting $t_0$ be sufficiently small, we have from  (\ref{23/05/16(11)}), (\ref{23/05/16(12)}) and (\ref{23/05/16(20)}) that
\begin{equation}
\label{23/05/16(50)}
\|x\|_{L^2([0, t_0], V)}\le M(t_0) (\|\phi_0\|_H + \|\phi_1\|_{L^2([-r, 0], V)}),
\end{equation}
where \[
M(t_0) :=(1- \sqrt{t_0}\|D_2\| - C_1(t_0)\|F_2\|t_0 )^{-1}(C_1(t_0)\|F_1\|+C_2(t_0))>0.\]
 In a similar manner, we can show the relation (\ref{27/04/16(17890)}).

 Last, by repeating the above argument on $[t_0, 2t_0]$, $[2t_0, 3t_0]$, $\ldots$, we can finally show the existence and uniqueness of a  solution $x$ of  (\ref{25/04/16(1)}) on $[0, T]$   satisfying the estimate (\ref{27/04/16(1)}) or (\ref{27/04/16(17890)}) for each $T\ge 0$.
The proof is thus complete.
\end{proof}

The following results give conditions on the initial data in order to obtain a solution which is more regular with respect to time or space variables.

\begin{theorem}
\label{29/04/16(3)}
Suppose that $(\phi_0, \phi_1)\in V\times W^{1, 2}([-r, 0], V)$ with
\[
\phi_0=\phi_1(0)-D\phi_1\in V\hskip 15pt \hbox{ and}\hskip 15pt A\phi_0 + F\phi_1\in H,\]
then the solution $x$ of (\ref{25/04/16(2)}) satisfies
\begin{equation}
\label{29/04/16(002)}
 x(\cdot) \in W^{1, 2}([-r, T], V),
\end{equation}
and
\begin{equation}
\label{29/04/16(2)}
 x(\cdot) -Dx_{\cdot} \in W^{1, 2}([-r, T], V)\cap W^{2, 2}([0, T], V^*)\subset C^1([0, T], H),
\end{equation}
for each $T\ge 0$.
\end{theorem}
\begin{proof}
In correspondence with the time $t_0\in [0, r]$ in the proof of Theorem \ref{29/04/16(1)}, let us consider the following closed subspace $\Sigma_0$ of $W^{1, 2}([-r, t_0], V)$ such that $x(0) = \phi_0 + D\phi_1$ and $x_0=\phi_1$ for $x\in \Sigma_0$. Define the same mapping $S$ as in (\ref{23/05/16(60)}) on $\Sigma_0$ by
\[
Sx(t) = Dx_t + e^{tA}\phi_0+ \int^t_0 e^{(t-s)A} Fx_sds\hskip 10pt \hbox{for any}\hskip 10pt t\in [0, t_0].\]
Once again, it is immediate from Theorem  \ref{23/05/16(10)} (iv) and (v) that $Sx(\cdot)\in W^{1, 2}([0, t_0], V)$ and $Sx(\cdot) - Dx_\cdot \in W^{1, 2}([0, t_0], V)\cap W^{2, 2}([0, t_0], V^*)$ for each $x\in \Sigma_0$.

We shall show that $S$ is a contraction from $\Sigma_0$ into itself for sufficiently small $t_0>0$. Indeed,  first we note that $Fx_t\in W^{1, 2}([0, T], V^*)$ for  $x\in \Sigma_0$ and $A\phi_0+ F\phi_1\in H$. Then for any $x,\,\bar x\in \Sigma_0$ and $t\in [0, t_0]$, we have
\[
\begin{split}
\|Sx-S\bar x&\|_{W^{1, 2}([0, t_0], V)}\\
\le&\,\, \|D_2(x_\cdot-\bar x_\cdot)\|_{W^{1, 2}([0, t_0], V)}+ \Big\|\int^\cdot_0 e^{(\cdot-s)A} F_2(x_s-\bar x_s)ds\Big\|_{W^{1, 2}([0, t_0], V)}\\
\le&\,\, \|D_2\|\sqrt{t_0}\|x(\cdot)-\bar x(\cdot)\|_{W^{1, 2}([-r, t_0], V)} + M_0 t_0 \|F_2\|\|x(\cdot)-\bar x(\cdot)\|_{W^{1, 2}([-r, t_0], V)}\\
=&\,\, \delta(t_0)\|x(\cdot)-\bar x(\cdot)\|_{W^{1, 2}([-r, t_0], V)},
\end{split}
\]
where $\delta(t_0)=  \|D_2\|\sqrt{t_0} +  M_0\|F_2\| t_0\to 0$ as $t_0\to 0$. Hence, map $S$ is a contraction in $\Sigma_0$ and  the equation (\ref{25/04/16(2)}) has a unique solution $x$ such that
\[
x(\cdot) \in W^{1, 2}([-r, t_0], V),\]
and, similarly,
\[
y(\cdot) = x(\cdot) - Dx_\cdot \in W^{1, 2}([-r, t_0], V)\cap W^{2, 2}([0, t_0], V^*)\subset C^1([0, t_0], H).\]
Moreover, we can proceed as in the proof of  Theorem \ref{29/04/16(1)} to get a solution $x$ in $[-r, T]$ and the relation (\ref{29/04/16(002)}) or (\ref{29/04/16(2)}) for all $T\ge 0$.  The proof is easily completed.
\end{proof}

  Let  $x(t)$, $t\ge -r$ (and $y(t)$, $t\ge 0$) denote the unique solution of system (\ref{25/04/16(1)})
 with initial $x(0)=\phi_0 + D\phi_1$ and $x_0=\phi_1$, $\phi=(\phi_0, \phi_1)\in {\cal H}$. We define a family of operators ${\cal S}(t):\, {\cal H}\to {\cal H}$, $t\ge 0$,  by
\begin{equation}
\label{05/06/06(1178)}
{\cal S}(t)\phi =(x(t)-Dx_t, x_t)\hskip 10pt \hbox{for any}\hskip 10pt \phi\in {\cal H}.
\end{equation}

\begin{theorem}
The family $t\to {\cal S}(t)$ is a strongly continuous semigroup on ${\cal H}$, i.e.,
\begin{enumerate}
\item[(i)] ${\cal S}(t)\in {\mathscr L}({\cal H})$ for each $t\ge 0$;
\item[(ii)] ${\cal S}(0)=I$, ${\cal S}(s+t)={\cal S}(s){\cal S}(t)$ for any $s,\,t\ge 0$;
\item[(iii)] $\lim_{t\to0^+}{\cal S}(t)\phi=\phi$ for each $\phi\in {\cal H}.$
\end{enumerate}
\end{theorem}
\begin{proof}
Let $y(t, \phi) = x(t, \phi) - Dx_t(\phi)$ for each $t\ge 0$, $\phi\in {\cal H}$ and the solution $x(t, \phi)$ of  (\ref{25/04/16(1)}). Then we have by virtue of Theorem \ref{29/04/16(1)}
 that for any $t\in [0, T]$,
\[
\begin{split}
\|{\cal S}(t)\phi\|_{\cal H}^2 &= \|x(t, \phi)- Dx_t(\phi)\|_H^2 + \int^0_{-r} \|x(t+\theta, \phi)\|^2_Vd\theta\\
&\le C_1\Big(\|x(\cdot)-Dx_\cdot\|^2_{L^2([0, T], V)} + \|x(\cdot)-Dx_\cdot\|^2_{W^{1, 2}([0, T], V^*)} + \int^T_{-r} \|x(t, \phi)\|^2_Vdt\Big)\\
&\le C_2(T)\big(\|\phi_0\|_H^2 + \|\phi_1\|^2_{L^2([-r, 0], V)}\big),
\end{split}
\]
where $C_1,\,C_2(T)>0$, which shows (i).  To show (ii),    it is easy to see from (\ref{25/04/16(2)}) that for any $t\ge s\ge 0$,
\[
\begin{split}
y(t-s, {\cal S}(s)\phi)&= e^{(t-s)A}({\cal S}\phi)_0 +\int^{t-s}_{0} e^{(t-s-u)A}Fx_{u}({\cal S}(s)\phi)du\\
&=  e^{tA}\phi_0 +\int^s_0 e^{(t-u)A}Fx_u(\phi)du+\int^{t}_{s} e^{(t-u)A}Fx_{u-s}({\cal S}(s)\phi)du.\\
\end{split}
\]
On the other hand, for $t\ge s$,
\begin{equation}
\label{02/11/06(478)}
\begin{split}
y(t, \phi) =  e^{tA}\phi_0 + \int^{s}_0 e^{(t-u)A}Fx_u(\phi)du+\int^t_{s} e^{(t-u)A}Fx_{u}(\phi)du,
\end{split}
\end{equation}
which immediately implies that for any $t\ge s\ge 0$,
 \begin{equation}
\label{28/04/16(1)}
\begin{split}
x(t-s, {\cal S}(s)\phi) - D&x_{t-s}({\cal S}(s)\phi) - \int^{t}_{s} e^{(t-u)A}Fx_{u-s}({\cal S}(s)\phi)du\\
& = x(t, \phi) - Dx_t(\phi) -\int^t_{s} e^{(t-u)A}Fx_{u}(\phi)du.
\end{split}
\end{equation}
Thus, by the uniqueness of solutions to (\ref{28/04/16(1)}),
it follows that
\begin{equation}
\label{06/10/07(22256)}
x(t-s, {\cal S}(s)\phi)=x(t, \phi),\hskip 20pt t\ge s,
\end{equation}
and further
\[
y(t-s, {\cal S}(s)\phi)=x(t-s, {\cal S}(s)\phi) -Dx_{t-s}({\cal S}(s)\phi) = x(t, \phi) -Dx_t(\phi)    =y(t, \phi)
\]
for any $t\ge s\ge 0$.
Hence,  $[{\cal S}(t-s){\cal S}(s)\phi]_0=[{\cal S}(t)\phi]_0$ and similarly  we can show that $[{\cal S}(t-s){\cal S}(s)\phi]_1=[{\cal S}(t)\phi]_1$ for any $t\ge s\ge 0$.

Finally, to show (iii) we notice that
\[
\|{\cal S}(t)\phi -\phi\|^2_{\cal H} = \|y(t)  -y(0)\|^2_H + \int^0_{-r}\|x(t+\theta)-\phi_1(\theta)\|^2_Vd\theta \to 0\hskip 10pt \hbox{as}\hskip 10pt t\to 0,\]
since $y\in C([0, T], H)$ and $x\in L^2([-r, T], V)$ for each $T\ge 0$. The proof is complete.
\end{proof}

The following theorem whose finite dimensional version was established in Ito and Tarn \cite{kitt85} gives a complete description of the generator ${\cal A}$ of semigroup $e^{t{\cal A}}$, $t\ge 0$.

\begin{theorem}
\label{06/03/2017(1)}
The generator ${\cal A}$ of the strongly continuous semigroup $e^{t{\cal A}}$, $t\ge 0$, is given by
\[
{\mathscr D}({\cal A}) = \Big\{(\phi_0, \phi_1)\in {\cal H}: \phi_1\in W^{1, 2}([-r, 0], V),\,\phi_0=\phi_1(0)-D\phi_1\in V,\,A\phi_0+F\phi_1\in H\Big\}\]
and for each $\phi=(\phi_0, \phi_1)\in {\mathscr D}({\cal A})$,
\[
{\cal A}\phi = (A\phi_0 + F\phi_1, \phi'_1)\in {\cal H}.\]
\end{theorem}

This theorem will result from the following several propositions according to Theorem 1.9 in Davies
\cite{ebd80}.

\begin{proposition}
For each $t\ge 0$, we have  ${\cal S}(t){{\mathscr D}({\cal A})}\subset {\mathscr D}({\cal A})$.
\end{proposition}
\begin{proof}
If $(\phi_0, \phi_1)\in {\mathscr D}({\cal A})$ and $(x(\cdot) - Dx_\cdot, x_\cdot)$ is the unique solution of equation
 (\ref{25/04/16(1)})  with initial data $(\phi_0, \phi_1)$, then for each $T\ge 0$, we get from Theorem \ref{29/04/16(3)} that $x\in W^{1, 2}([-r, T], V)$ and
\[
x(\cdot)-Dx_\cdot \in W^{1, 2}([0, T], V)\cap W^{2,2}([0, T], V^*)\subset C^1([0, T], H).\]
 Therefore, for each $t\ge 0$, we have  $x_t\in W^{1, 2}([-r, 0], V)$ and
\[
A(x(t) -Dx_t)  + Fx_t  = (x(t) -Dx_t)'\in H \hskip 15pt \hbox{and}\hskip 15pt x'_t\in L^2([-r, 0], V),\]
 a fact which immediately implies that $(x(t)- Dx_t, x'_t)\in {\mathscr D}({\cal A})$ for each $t\ge 0$. The proof is thus  complete.
\end{proof}

\begin{proposition}
The domain ${\mathscr D}({\cal A})$ is dense in ${\cal H}$.
\end{proposition}
\begin{proof}
Since ${\cal S}(t)$ is strongly continuous, we have
\[
\lim_{\varepsilon\to 0}\varepsilon^{-1}\int^\varepsilon_0 {\cal S}(t)\phi dt = \phi\hskip 15pt \hbox{for each}\hskip 10pt \phi\in {\cal H}.\]
Hence, it suffices to prove that for each $\varepsilon>0$,
\[
\int^\varepsilon_0 {\cal S}(t)\phi dt = \Big(\int^\varepsilon_0 (x(t)- Dx_t)dt, \int^\varepsilon_0 x_t(\cdot)dt\Big)\in {\mathscr D}({\cal A}),\]
where $x$  is the unique solution of equation (\ref{25/04/16(1)}).

Since $x\in L^2([-r, T], V)$ for each $T>0$, we have that for any $\theta\in [-r, 0]$,
\[
\Big(\int^\varepsilon_0 x_sds\Big)(\theta)= \int^\varepsilon_0 x(s+\theta)ds = \int^{\varepsilon+\theta}_0 x(u)du - \int^\theta_0 x(u)du,\]
which immediately implies
\[
\Big(\int^\varepsilon_0 x_tdt\Big)(\cdot) \in W^{1, 2}([-r, 0], V)\hskip 10pt \hbox{and}\hskip 10pt
\Big(\int^\varepsilon_0 x_tdt\Big)(0)= \int^\varepsilon_0 x(t)dt.\]
Hence, we have
\[
\begin{split}
 \int^\varepsilon_0 x(t)dt - \int^\varepsilon_0 Dx_tdt = \Big(\int^\varepsilon_0 x_t dt\Big)(0) -D\int^\varepsilon_0 x_tdt.
\end{split}
\]
Further, since $x\in L^2([0, \varepsilon], V)\cap W^{1, 2}([0, \varepsilon], V^*)\subset C([0, \varepsilon], H)$, we also obtain
\[
\begin{split}
A\int^\varepsilon_0 (x(t)- Dx_t)dt&\, + F_1\int^\varepsilon_0 x(t-r)dt + F_2\int^\varepsilon_0 x_tdt\\
&= \int^\varepsilon_0 \frac{d(x(t)-Dx_t)}{dt}dt = x(\varepsilon)-Dx_\varepsilon -\phi_0\in H.
\end{split}
\]
The proof is thus complete.
\end{proof}

\begin{proposition}
If $\phi=(\phi_0, \phi_1)\in {\mathscr D}({\cal A})$, then we have
\[
\lim_{t\downarrow 0}\frac{{\cal S}(t)\phi - \phi}{t} = {\cal A}\phi = (A\phi_0+ F\phi_1, \phi'_1).\]
\end{proposition}
\begin{proof}
Let $\phi=(\phi_0, \phi_1)\in {\mathscr D}({\cal A})$. Then it is known by Theorem \ref{29/04/16(3)} that the  corresponding solution $x$ of   (\ref{25/04/16(1)}) with initial $\phi$ is in $C^1([0, T], H)$ for each $T>0$. Hence, by the equality (\ref{05/06/06(1178)}) we have
\begin{equation}
\label{24/05/16(1)}
\lim_{t\downarrow 0}\Big\|\frac{x(t)- Dx_t-\phi_0}{t} - x'(0) - Dx'_0\Big\|_H =  \lim_{t\downarrow 0}\Big\|\frac{x(t)-Dx_t-\phi_0}{t} - A\phi_0 - F\phi_1\Big\|_H=0.
\end{equation}
On the other hand, as $x\in W^{1, 2}([-r, T], V)$ for each $T>0$ according to Theorem  \ref{29/04/16(3)}, we can write for $\theta\in [-r, 0]$ and $t>0$ that
\begin{equation}
\label{24/05/16(2)}
\frac{x_t(\theta)-\phi_1(\theta)}{t} = \frac{1}{t}\int^{t+\theta}_\theta x'(s)ds.
\end{equation}
Since $x'\in L^2([-r, T], V)$ for each $T>0$, we have
\[
\lim_{t\downarrow 0}\int^0_{-r} \Big\|\frac{1}{t}\int^{\theta+t}_\theta x'(s)ds -x'(\theta)\Big\|^2_V=0.\]
Therefore, it is easy to get from (\ref{24/05/16(2)}) that
\begin{equation}
\label{24/05/16(3)}
\lim_{t\downarrow 0}\Big\|\frac{x_t-\phi_1}{t}-\phi'_1\Big\|_{L^2([-r, 0], V)}=0.
\end{equation}
The conclusion follows from (\ref{24/05/16(1)}) and (\ref{24/05/16(3)}) and the proof is thus complete.
\end{proof}

\begin{proposition}
The map ${\cal A}:\,  {\mathscr D}({\cal A})\subset {\cal H}\to {\cal H}$  is a closed operator.
\end{proposition}
\begin{proof}
Suppose that there exist a sequence $\{(\phi_{0, n}, \phi_{1, n})\}_{n\ge 1}\in {\mathscr D}({\cal A})$ such that
\begin{equation}
\label{24/05/16(5)}
 (\phi_{0, n}, \phi_{1, n})\to (\phi_0, \phi_1)\hskip 10pt \hbox{as}\hskip 10pt n\to\infty\hskip 10pt \hbox{in}\hskip 10pt {\cal H}
\end{equation}
and
\begin{equation}
\label{24/05/16(6)}
{\cal A} (\phi_{0, n}, \phi_{1, n})\to (\psi_0, \psi_1)\hskip 10pt \hbox{as}\hskip 10pt n\to\infty\hskip 10pt \hbox{in}\hskip 10pt {\cal H}.
\end{equation}
We need to show $(\phi_0, \phi_1)\in {\mathscr D}({\cal A})$ and ${\cal A}(\phi_0, \phi_1)=  (\psi_0, \psi_1)$.

Indeed, since $\phi_{1, n}\in W^{1, 2}([-r, 0], V)$ for each $n\ge 1$, it follows from (\ref{24/05/16(5)}) and (\ref{24/05/16(6)}) that
\[
\lim_{n\to\infty} \|\phi_{1, n}-\phi_1\|_{L^2([-r, 0], V)}=0, \hskip 20pt \lim_{n\to\infty} \|\phi'_{1, n}-\psi_1\|_{L^2([-r, 0], V)}=0.\]
Hence, $\phi_1\in W^{1, 2}([-r, 0], V)$ and $\phi'_1=\psi_1$. This implies
\begin{equation}
\label{24/05/16(20)}
\phi_{1, n}\to \phi_1\hskip 10pt \hbox{in}\hskip 10pt W^{1, 2}([-r, 0], V)\hskip 10pt \hbox{as}\hskip 10pt n\to\infty,
\end{equation}
 and by Sobolev's imbedding theorem,
\begin{equation}
\label{24/05/16(21)}
\phi_{1, n}(0)\to \phi_1(0)\hskip 10pt \hbox{ and}\hskip 10pt \phi_{1, n}(-r)\to \phi_1(-r)\hskip 10pt \hbox{in}\hskip 10pt V\hskip 10pt \hbox{as}\hskip 10pt n\to\infty.
\end{equation}
 From (\ref{24/05/16(5)}), we thus have the following equalities
\begin{equation}
\label{24/05/16(22)}
\begin{split}
\phi_{1}(0)=\lim_{n\to\infty}\phi_{1, n}(0)&= \lim_{n\to\infty}(\phi_{0, n}-D_1\phi_{1, n}(-r)-D_2\phi_{1, n})\\
&=\phi_0-D_1\phi_1(-r)-D_2\phi_1\\
 &=\phi_{0}-D\phi_{1}\hskip 10pt \hbox{in}\hskip 10pt H.
\end{split}
\end{equation}
Further, by virtue of (\ref{24/05/16(20)}) and (\ref{24/05/16(21)}) we have
\begin{equation}
\label{24/05/16(25)}
A\phi_{0, n} + F_1\phi_{1, n}(-r) + F_2\phi_{1, n}\to A\phi_{0} + F_1\phi_{1}(-r) + F_2\phi_{1}\hskip 10pt \hbox{in}\hskip 10pt V^*\hskip 10pt \hbox{as}\hskip 10pt n\to\infty.
\end{equation}
On the other hand, we have from (\ref{24/05/16(6)}) that
\begin{equation}
\label{24/05/16(26)}
A\phi_{0, n} + F_1\phi_{1, n}(-r) + F_2\phi_{1, n}\to \psi_0\hskip 10pt  \hbox{in}\hskip 10pt H\hskip 10pt \hbox{as}\hskip 10pt n\to\infty.
\end{equation}
Hence, from (\ref{24/05/16(25)}) and (\ref{24/05/16(26)}) we obtain
\[
\psi_0 = A\phi_{0} + F_1\phi_{1}(-r) + F_2\phi_{1}\in H\]
as desired. Hence, the proof is complete now.
\end{proof}

\section{\large Semigroup and Resolvent}

For each $\lambda\in {\mathbb C}$, we define a linear operator $D(e^{\lambda\cdot}):\, V\to V$ by
\[
D(e^{\lambda\cdot})x = D(e^{\lambda\cdot}x) \hskip 15pt \hbox{for any}\hskip 10pt x\in V.\]
Then it is easy to see that $D(e^{\lambda\cdot})\in {\mathscr L}(V)$. Indeed, for any $x\in V$,
\[
\begin{split}
  \|D(e^{\lambda\cdot})x\|_V &= \|D(e^{\lambda\cdot}x)\|_V\\
   &\le \|D_1(e^{-\lambda r}x)\|_V + \|D_2\|\Big(\int^0_{-r} \|e^{\theta\lambda}x\|^2_Vd\theta\Big)^{1/2}\\
& \le \|D_1\|e^{|\lambda| r}\|x\|_V + \|D_2\|\sqrt{r}e^{|\lambda|r}\|x\|_V.
\end{split}
\]
In a similar way, one can show that $F(e^{\lambda\cdot})\in {\mathscr L}(V, V^*)$.
 For each  $\lambda\in {\mathbb C}$,
we define a linear operator $\Delta(\lambda):\, V\to V^*$ by
\begin{equation}
\label{06/11/16(1)}
\Delta(\lambda) = \Delta(\lambda, A, D, F) = (\lambda I - A)(I - D(e^{\lambda\cdot})) - F(e^{\lambda\cdot})\in {\mathscr L}(V, V^*).
\end{equation}
The resolvent set $\rho(A, D, F)$ is defined as the family of all values $\lambda$ in ${\mathbb C}$ for which the operator $\Delta(\lambda, A, D, F)$ has a bounded inverse  $\Delta(\lambda, A, D, F)^{-1}$ on $V^*$. The operator $\Delta(\lambda, A, D, F)^{-1}$ is called the {\it resolvent\/} of $(A, D, F)$.

The following proposition can be used to establish useful relations between the  resolvent  $\Delta(\lambda, A, D, F)^{-1}$ and resolvent $(\lambda I - {\cal A})^{-1}$ of ${\cal A}$.
\begin{proposition}
\label{14/08/2013(70)}
Let $\lambda\in {\mathbb C}$ and $\psi=(\psi_0, \psi_1)\in {\cal H}$. If $\phi=(\phi_0, \phi_1)\in {\mathscr D}({\cal A)}$ satisfies
\begin{equation}
\label{28/07/2013(2)}
\lambda\phi -{\cal A}\phi = \psi,
\end{equation}
then
\begin{equation}
\label{28/07/2013(5)}
\phi_1(\theta) = e^{\lambda\theta}\phi_1(0) + \int^0_{\theta}e^{\lambda(\theta-\tau)}\psi_1(\tau)d\tau,\hskip 20pt -r\le \theta\le 0,
\end{equation}
and
\begin{equation}
\label{28/07/2013(1)}
\Delta(\lambda) \phi_1(0) =(\lambda I-A)D\Big(\int^0_\cdot e^{\lambda(\cdot-\tau)}\psi_1(\tau)d\tau\Big) + F\Big(\int^0_{\cdot}e^{\lambda(\cdot-\tau)}\psi_1(\tau)d\tau\Big) +\psi_0.
\end{equation}
Conversely, if $ \phi_1(0)\in V$ satisfies the equation (\ref{28/07/2013(1)}) and let $\phi_0 =\phi_1(0)-D\phi_1$ where
\begin{equation}
\label{28/07/2013(20)}
\phi_1(\theta) = e^{\lambda\theta} \phi_1(0) +  \int^0_{\theta}e^{\lambda(\theta-\tau)}\psi_1(\tau)d\tau,\hskip 20pt -r\le \theta\le 0,
\end{equation}
then we have that $\phi_1\in W^{1,2}([-r, 0], V)$, $\phi=(\phi_0, \phi_1)\in {\mathscr D}({\cal A})$ and $\phi$ satisfies (\ref{28/07/2013(2)}).
\end{proposition}
\begin{proof}
The equation (\ref{28/07/2013(2)}) can be equivalently written as
\begin{equation}
\label{28/07/2013(10)}
\lambda \phi_1(0) - \lambda D\phi_1 -A\phi_1(0) +AD\phi_1 -F\phi_1=\psi_0,
\end{equation}
and
\begin{equation}
\label{28/07/2013(4)}
\lambda \phi_1(\theta) -{d\phi_1(\theta)}/{d\theta}= \psi_1(\theta)\hskip 15pt \hbox{for}\,\,\,\,\theta\in [-r, 0].
\end{equation}
It is easy to see that (\ref{28/07/2013(4)}) is equivalent to (\ref{28/07/2013(5)}). Hence, if (\ref{28/07/2013(2)}) holds, we get  $\phi_1(0)\in V$ and by virtue of (\ref{28/07/2013(10)}) and (\ref{28/07/2013(5)}), we have
\[
\begin{split}
\Delta(&\lambda, A, D, F)\phi_1(0) \\
=&\,\, \lambda\phi_1(0) -A\phi_1(0) + AD(e^{\lambda\cdot})\phi_1(0) -\lambda D(e^{\lambda\cdot})\phi_1(0) -F(e^{\lambda\cdot})\phi_1(0)\\
=&\,\, \lambda D\phi_1 -AD\phi_1 +F\phi_1 +\psi_0  + AD(e^{\lambda\cdot})\phi_1(0) -\lambda D(e^{\lambda\cdot})\phi_1(0) -F(e^{\lambda\cdot})\phi_1(0)\\
=&\,\, \lambda D(e^{\lambda\cdot})\phi_1(0) + \lambda D\Big(\int^0_\cdot e^{\lambda(\cdot-\tau)}\psi_1(\tau)d\tau\Big) -AD(e^{\lambda\cdot})\phi_1(0) -AD\Big(\int^0_\cdot e^{\lambda(\cdot-\tau)}\psi_1(\tau)d\tau\Big)\\
&\,\, + F(e^{\lambda\cdot})\phi_1(0) + F\Big(\int^0_\cdot e^{\lambda(\cdot-\tau)} \psi_1(\tau)d\tau\Big) +\psi_0\\
&\,\, +AD(e^{\lambda\cdot})\phi_1(0) -\lambda D(e^{\lambda\cdot})\phi_1(0) - F(e^{\lambda\cdot})\phi_1(0)\\
=&\,\, (\lambda I-A)D\Big(\int^0_\cdot e^{\lambda(\cdot-\tau)}\psi_1(\tau)d\tau\Big) + F\Big(\int^0_{\cdot}e^{\lambda(\cdot-\tau)}\psi_1(\tau)d\tau\Big) +\psi_0
\end{split}
\]
which is the equality (\ref{28/07/2013(1)}).

Conversely, if $\phi_1(0)\in V$,  then by a simple calculation it is easy to see that $\phi_1$, defined by (\ref{28/07/2013(20)}), belongs to $W^{1, 2}([-r, 0], V)$. In addition, let $\phi_0=\phi_1(0)-D\phi_1\in V$ and  assume that (\ref{28/07/2013(1)}) holds true. Then from (\ref{28/07/2013(1)}) and  (\ref{28/07/2013(20)}), we get
\begin{equation}
\label{28/07/2013(50)}
\begin{split}
\lambda \phi_0 -A\phi_0 &= \lambda \phi_1(0)-A\phi_1(0)- \lambda D\phi_1 + AD\phi_1\\
&= (\lambda I-A) D\Big(\int^0_\cdot e^{\lambda(\cdot-\tau)}\psi_1(\tau)d\tau\Big) + F\Big(\int^0_\cdot e^{\lambda(\cdot-\tau)}\psi_1(\tau)d\tau\Big) + \psi_0\\
&\,\,\,\,\,\,+ (\lambda I-A)D(e^{\lambda\cdot})\phi_1(0) + F(e^{\lambda\cdot})\phi_1(0)-\lambda D\phi_1+ AD\phi_1\\
&=(\lambda I-A)D\phi_1 +F\phi_1 +\psi_0-\lambda D\phi_1+ AD\phi_1\\
&= F\phi_1 + \psi_0.
\end{split}
\end{equation}
Therefore, $A\phi_0 + F\phi_1 =\lambda\phi_0 -\psi_0\in H$, which is  the first coordinate  relation of  (\ref{28/07/2013(2)}). The second coordinate equality of  (\ref{28/07/2013(2)}) is obvious. The proof is thus complete.
\end{proof}

\begin{proposition}
For each $\lambda\in {\mathbb C}$, the mapping $\Delta(\lambda)$ is injective if and only if $\lambda I -{\cal A}$ is injective.
\end{proposition}
\begin{proof}
Let $\psi_0=0$, $\psi_1(\cdot)\equiv 0$ in Proposition \ref{14/08/2013(70)}. If $\Delta(\lambda)x=0$ with $x\not= 0$, then we may take $\phi=(\phi_0, \phi_1)$ where
\[
\phi_0 =x-D\phi_1,\hskip 15pt \phi_1(\theta) = e^{\lambda\theta}x,\hskip 15pt \theta\in [-r, 0].\]
Hence, $\phi=(\phi_0, \phi_1)\not= 0$ and
\[
(\lambda I -{\cal A})\phi=0.\]
Conversely, suppose that there exists a non-zero $\phi=(\phi_0, \phi_1)\in {\mathscr D}({\cal A})$ satisfying
$(\lambda I - {\cal A})\phi=0$. Then it can't happen that $\phi_1(\theta)=0$ for all $\theta\in [-r, 0]$ since $\phi_0=\phi_1(0)-D\phi_1=0$ otherwise. Hence, there exists a value $\theta\in [-r, 0]$ such that
\[
\phi_1(0) = e^{-\lambda\theta}\phi_1(\theta)\not= 0.\]
Let $x=\phi_1(0)\not= 0$, then by virtue of (\ref{28/07/2013(1)}) we have
\[
\Delta(\lambda)x=0.\]
The proof is thus complete.
\end{proof}

\begin{proposition}
Suppose that  $\Delta(\lambda)V= V^*$ for some $\lambda\in {\mathbb C}$, then
\[
(\lambda I -{\cal A}){\mathscr D}({\cal A})={\cal H}.\]
\end{proposition}
\begin{proof}
For  $\lambda\in {\mathbb C}$ and $\psi=(\psi_0, \psi_1)\in {\cal H}$, since  $\Delta(\lambda)V= V^*$, there exists an element $\phi_1(0)\in V$ such that
\[
\Delta(\lambda)\phi_1(0) = (\lambda I -A)D\Big(\int^0_{\cdot}e^{\lambda(\cdot-\tau)}\psi_1(\tau)d\tau\Big) + F\Big(\int^0_{\cdot}e^{\lambda(\cdot-\tau)}\psi_1(\tau)d\tau\Big)+\psi_0\in V^*.\]
Let $\phi_0=\phi_1(0)-D\phi_1\in V$ where $\phi_1$ is given by
\[
\phi_1(\theta) = e^{\lambda\theta}\phi_1(0) + \int^0_{\theta}e^{\lambda(\cdot-\tau)}\psi_1(\tau)d\tau,\hskip 15pt \theta\in [-r, 0].\]
Then by Proposition \ref{14/08/2013(70)}, we have that $\phi_1\in W^{1, 2}([-r, 0], V)$, $\phi=(\phi_0, \phi_1)\in {\mathscr D}({\cal A})$ and $\phi$ satisfies $\lambda\phi - {\cal A}\phi=\psi\in {\cal H}$ as desired. The proof is complete now.
\end{proof}

\begin{proposition}
Let $\lambda\in {\mathbb C}$. If there exists $C_1>0$ such that
\begin{equation}
\label{05/09/16}
\|x\|_V \le C_1\|\Delta(\lambda)x\|_{V^*}\hskip 10pt \hbox{for each}\hskip 10pt x\in V,
\end{equation}
 then there exists a constant $C_2>0$ such that
\[
\|\phi\|_{\cal H} \le C_2\|(\lambda I -{\cal A})\phi\|_{\cal H}\hskip 10pt \hbox{for each}\hskip 10pt \phi\in {\mathscr D}({\cal A}).\]
\end{proposition}
\begin{proof}
First note that for any $\phi_1(0)\in V$ and $\psi_1\in L^2([-r, 0], V)$, the function
\begin{equation}
\label{05/09/16(2)}
\phi_1(\theta) = e^{\lambda\theta}\phi_1(0) + \int^0_\theta e^{\lambda(\theta-\tau)}\psi_1(\tau)d\tau,\hskip 15pt \theta\in [-r, 0],
\end{equation}
satisfies the relations
\begin{equation}
\label{07/09/16(3)}
\begin{split}
\|\phi_1\|_{L^2([-r, 0], V)} &\le \|e^{\lambda\cdot}\phi_1(0)\|_{L^2([-r, 0], V)} + \Big\|\int^0_\cdot e^{\lambda(\cdot-\tau)}\psi_1(\tau)d\tau\Big\|_{L^2([-r, 0], V)}\\
&\le c_1(\lambda)\|\phi_1(0)\|_V + c_2(\lambda)\|\psi_1\|_{L^2([-r, 0], V)},
\end{split}
\end{equation}
where $c_1(\lambda),\,c_2(\lambda)>0$.

For $\lambda\in {\mathbb C}$, we define a linear operator $\Sigma_\lambda:\,L^2([-r, 0], V)\to V^*$ by
\begin{equation}
\label{07/09/2016(10)}
\Sigma_\lambda\varphi = (\lambda I -{A})D\Big(\int^0_\cdot e^{\lambda(\cdot-\tau)}\varphi(\tau)d\tau\Big) + F\Big(\int^0_\cdot e^{\lambda(\cdot-\tau)}\varphi(\tau)d\tau\Big),\hskip 15pt \forall\,\varphi\in L^2([-r, 0], V).
\end{equation}
Then it is easy to see that $\Sigma_\lambda\in {\mathscr L}(L^2([-r, 0], V), V^*).$ Indeed, for each $\varphi\in L^2([-r, 0], V)$ we have
\begin{equation}
\label{07/09/16(1)}
\begin{split}
\|\Sigma_\lambda\varphi\|_{V^*} \le &\,\, \Big\|(\lambda I -A)D\Big(\int^0_\cdot e^{\lambda(\cdot-\tau)}\varphi(\tau)d\tau\Big)\Big\|_{V^*} + \Big\|F\Big(\int^0_\cdot e^{\lambda(\cdot-\tau)}\varphi(\tau)d\tau\Big)\Big\|_{V^*}\\
\le &\,\, \|\lambda I -A\|_{{\mathscr L}(V, V^*)}\|D\|_{{\mathscr L}(L^2([-r, 0], V), V)}\Big\|\int^0_\cdot e^{\lambda(\cdot-\tau)}\varphi(\tau)d\tau\Big\|_{L^2([-r, 0], V)}\\
&\,\, + \|F\|_{{\mathscr L}(L^2([-r, 0], V), V^*)}\Big\|\int^0_\cdot e^{\lambda(\cdot-\tau)}\varphi(\tau)d\tau\Big\|_{L^2([-r, 0], V)}\\
\le &\,\, c_2(\lambda)\Big[\|\lambda I -A\|_{{\mathscr L}(V, V^*)}\|D\|_{{\mathscr L}(L^2([-r, 0], V), V)} + \|F\|_{{\mathscr L}(L^2([-r, 0], V), V^*)}\Big]\|\varphi\|_{L^2([-r, 0], V)}\\
=: &\,\, c_3(\lambda)\|\varphi\|_{L^2([-r, 0], V)},
\end{split}
\end{equation}
where $c_3(\lambda)>0$.

For $\phi=(\phi_0, \phi_1)= (\phi_1(0)-D\phi_1, \phi_1)\in {\mathscr D}({\cal A})$, we set $\psi= \lambda \phi - {\cal A}\phi$ as in (\ref{28/07/2013(2)}). Then by  virtue of (\ref{28/07/2013(1)}), (\ref{28/07/2013(20)}) and (\ref{07/09/16(1)}), the condition (\ref{05/09/16}) implies that
\begin{equation}
\label{05/09/16(10)}
\begin{split}
\|\phi_1(0)\|_V &\le C_1\|\Delta(\lambda)\phi_1(0)\|_{V^*}\\
&= C_1\|\Sigma_\lambda \psi_1 + \psi_0\|_{V^*}\\
&\le C_1\cdot c_3(\lambda)\|\psi_1\|_{L^2([-r, 0], V)} + C_1\|\psi_0\|_{V^*},
\end{split}
\end{equation}
and, in addition to (\ref{05/09/16(2)}) and (\ref{05/09/16(10)}), that
 \begin{equation}
\label{05/09/16(11)}
\begin{split}
\|D\phi_1\|_V &\le \|De^{\lambda\cdot}\phi_1(0)\|_V + \Big\|D\Big(\int^0_\cdot e^{\lambda(\cdot-\tau)}\psi_1(\tau)d\tau\Big)\Big\|_V\\
&\le \Big(\|D_1\|e^{|\lambda|r} + \|D_2\|\sqrt{r}e^{|\lambda|r}\Big)\|\phi_1(0)\|_V + c_2(\lambda)\|D\|_{{\mathscr L}(L^2([-r, 0], V), V)}\cdot \|\psi_1\|_{L^2([-r, 0], V)}\\
&\le \Big[C_1c_3(\lambda)\Big(\|D_1\|e^{|\lambda|r} + \|D_2\|\sqrt{r}e^{|\lambda|r}\Big) + c_2(\lambda)\|D\|_{L^2([-r, 0], V), V)}\Big]\|\psi_1\|_{L^2([-r, 0], V)}\\
&\,\,\,\,\, + C_1\Big(\|D_1\|e^{|\lambda|r} + \|D_2\|\sqrt{r}e^{|\lambda|r}\Big)\|\psi_0\|_{V^*}\\
&=: c_4(\lambda) \|\psi_1\|_{L^2([-r, 0], V)} + c_5(\lambda)\|\psi_0\|_{V^*}.
\end{split}
\end{equation}
Combining (\ref{05/09/16(10)}) and (\ref{05/09/16(11)}), we thus have
\begin{equation}
\label{05/09/16(30)}
\begin{split}
\|\phi_0\|_V &\le \|\phi_1(0)\|_V + \|D\phi_1\|_V\\
& \le (C_1\cdot c_3(\lambda) + c_4(\lambda)) \|\psi_1\|_{L^2([-r, 0], V)} + (C_1+c_5(\lambda))\|\psi_0\|_{V^*}.
\end{split}
\end{equation}
Now from (\ref{07/09/16(3)}), (\ref{05/09/16(10)}), (\ref{05/09/16(30)}) and the fact that $\|\cdot\|_H\le \nu\|\cdot\|_V$, $\nu>0$, it follows for $\phi\in {\mathscr D}({\cal A})$ that
\begin{equation}
\label{05/09/16(60)}
\begin{split}
\|\phi\|_{\cal H} &\le \sqrt{2}(\|\phi_0\|_H + \|\phi_1\|_{L^2([-r, 0], V)})\\
&\le  \sqrt{2}\beta(C_1\cdot c_3(\lambda) + c_4(\lambda))\|\psi_1\|_{L^2([-r, 0], V)} + \sqrt{2}\nu (C_1+c_5(\lambda))\|\psi_0\|_{V^*}\\
&\,\,\,\,\, + \sqrt{2}c_1(\lambda)\|\phi_1(0)\|_V + \sqrt{2}c_2(\lambda)\|\psi_1\|_{L^2([-r, 0], V)}\\
&\le \Big( \sqrt{2}\nu (C_1\cdot c_3(\lambda) + c_4(\lambda)) + \sqrt{2}c_2(\lambda) + \sqrt{2}c_1(\lambda)C_1c_3(\lambda)\Big)\|\psi_1\|_{L^2([-r, 0], V)}\\
&\,\,\,\,\, + \Big(\sqrt{2}c_1(\lambda)C_1^2c_3(\lambda) + \sqrt{2}\nu (C_1+c_5(\lambda)\Big)\|\psi_0\|_{V^*}\\
&=: c_6(\lambda)\|\psi_1\|_{L^2([-r, 0], V)} + c_7(\lambda)\|\psi_0\|_{V^*}.
\end{split}
\end{equation}
Since $\|\cdot\|_{V^*}\le \nu\|\cdot\|_H$ for some $\nu>0$, it further follows from (\ref{05/09/16(60)}) that
\[
\begin{split}
\|\phi\|^2_{\cal H} &\le 2c^2_6(\lambda)\|\psi_1\|^2_{L^2([-r, 0], V)} + 2c^2_7(\lambda)\nu\|\psi_0\|^2_H\\
& = 2(c^2_6(\lambda)+c^2_7(\lambda))\|\psi\|^2_{\cal H}=2(c^2_6(\lambda)+c^2_7(\lambda))\|\lambda \phi -{\cal A}\phi\|^2_{\cal H},
\end{split}
\]
with $c_6(\lambda)>0,\, c_7(\lambda)>0$. The proof is thus complete.
\end{proof}

\begin{proposition}
Let $\lambda\in {\mathbb C}$. If $(\lambda -{\cal A}){\mathscr D}({\cal A})$ is dense in ${\cal H}$, then $\Delta(\lambda)V$ is dense in $V^*$.
\end{proposition}
\begin{proof}
For any $\psi_0\in H$, let $\psi=(\psi_0, 0)\in {\cal H}$. Then by assumption for each $\varepsilon>0$, there exist $\psi_\varepsilon=(\psi_{0, \varepsilon}, \psi_{1, \varepsilon})\in {\cal H}$ and $\phi_\varepsilon= (\phi_{0, \varepsilon}, \phi_{1, \varepsilon})\in {\mathscr D}({\cal A})$ such that
\begin{equation}
\label{11/09/16(2)}
\lambda\phi_\varepsilon -{\cal A}\phi_\varepsilon=\psi_\varepsilon\hskip 20pt \hbox{and}\hskip 20pt \|\psi_\varepsilon -\psi\|_{\cal H}<\varepsilon.
\end{equation}
From (\ref{28/07/2013(1)}), we have
\[
\Delta(\lambda)\phi_{1, \varepsilon}(0)= \Sigma_\lambda\psi_{1, \varepsilon}+ \psi_{0, \varepsilon},\]
where the operator $\Sigma_\lambda$ is given in (\ref{07/09/2016(10)}). Therefore,
\[
\|\psi_0- \Sigma_\lambda \psi_{1, \varepsilon} - \psi_{0, \varepsilon}\|_{V^*}\le \|\psi_0-\psi_{0, \varepsilon}\|_{V^*} + \|\Sigma_\lambda\psi_{1, \varepsilon}\|_{V^*},\]
and from (\ref{11/09/16(2)}) and the relation $\|\cdot\|_{V^*}\le \nu \|\cdot\|_H$ for some $\nu>0$, we have
\[
\|\psi_0-\psi_{0, \varepsilon}\|_{V^*} \le \nu \|\psi_0-\psi_{0, \varepsilon}\|_H <\varepsilon \nu,
\]
and again from  (\ref{11/09/16(2)}), it follows that
\[
\|\psi_{1, \varepsilon}\|_{L^2([-r, 0], V)}<\varepsilon.\]
Hence, we have the relations
\[
\|\Sigma_\lambda \psi_{1, \varepsilon}\|\le \|\Sigma_\lambda\|\cdot\|\psi_{1, \varepsilon}\|_{L^2([-r, 0], V)}\le \|\Sigma_\lambda\|\varepsilon.\]
In other words, we just show that for any $\psi_0\in H$ and $\varepsilon>0$, there exists $\Sigma_\lambda\psi_{1, \varepsilon} + \psi_{0, \varepsilon}\in \Delta(\lambda)V$ such that
\[
\|\psi_0 -(\Sigma_\lambda \psi_{1, \varepsilon} + \psi_{0, \varepsilon})\|_{V^*} \le \varepsilon(\nu + \|\Sigma_\lambda\|).\]
Since $H$ is dense in $V^*$, the desired result is thus proved.
\end{proof}

\section{\large Spectrum and Stationary Solution}

 First, let us consider the following deterministic functional differential equation of neutral type in $V^*$,
\begin{equation}
\label{11/08/2013(303)}
\begin{cases}
d\Big(y(t) - \alpha_1 y(t-r) - \displaystyle\int^0_{-r}\gamma(\theta)y(t+\theta)d\theta\Big)  =A\Big(y(t) -  \alpha_1 y(t-r) - \displaystyle\int^0_{-r}\gamma(\theta)y(t+\theta)d\theta\Big)dt\\
\hskip 200pt   +\, \alpha_2 A y(t-r)dt +\displaystyle\int^0_{-r} \beta(\theta)A y(t+\theta)d\theta dt,\,\,\,\,t\ge 0,\\
y(0)=\phi_0,\,\,\,\,y_0=\phi_1,\,\,\,\phi=(\phi_0, \phi_1)\in {\cal H},
\end{cases}
\end{equation}
where $\alpha_1,\,\alpha_2\in {\mathbb R}$ and $\beta,\,\gamma\in L^2([-r, 0]; {\mathbb R})$. By virtue of (\ref{05/06/06(1178)}) and Theorem \ref{06/03/2017(1)}, the equation (\ref{11/08/2013(303)}) can be equivalently lifted up into a deterministic equation without time delay
\begin{equation}
\begin{cases}
dY(t) ={\cal A}Y(t)dt, \,\,\,\,t\ge 0,\\
Y(0)=\phi\in {\cal H},
\end{cases}
\end{equation}
where ${\cal A}$ is the generator given in Theorem \ref{06/03/2017(1)}
and $Y(t)=(y(t), y_t)$ for all $t\ge 0$. On the other hand,  the characteristic operator $\Delta(\lambda)$  defined in
(\ref{06/11/16(1)})
 is given in this case  by
\[
\begin{split}
\Delta(\lambda)x =&\,\, \Big(1 - \alpha_1 e^{-\lambda r} - \displaystyle \int^0_{-r}\gamma(\theta)e^{\lambda \theta}d\theta \Big)(\lambda I -A)x -\alpha_2 e^{-\lambda r}Ax - \int^0_{-r} \beta(\theta)e^{\lambda \theta}d\theta Ax\\
=:&\,\, m(\lambda)x -n(\lambda)Ax\hskip 15pt \hbox{for each}\hskip 15pt \lambda\in {\mathbb C},\,\,x\in V,
\end{split}
\]
where
\begin{equation}
\label{12/09/16(1)}
m(\lambda) = \lambda\Big(1 -\alpha_1 e^{-\lambda r}-\int^0_{-r}\gamma(\theta)e^{\lambda\theta}d\theta\Big),\hskip 15pt \lambda\in {\mathbb C},
\end{equation}
and
\begin{equation}
\label{23/08/13(212)}
n(\lambda) = 1- \alpha_1 e^{-\lambda r} - \int^0_{-r} \gamma(\theta)e^{\lambda \theta}d\theta + \alpha_2 e^{-\lambda r} + \int^0_{-r} \beta(\theta) e^{\lambda\theta}d\theta,\hskip 15pt \lambda\in {\mathbb C}.
\end{equation}
\begin{definition} \rm
The {\it point spectrum\/} $\sigma_p(\Delta)$  is defined to be the set
\[
\sigma_p(\Delta) =\{\lambda\in {\mathbb C}:\, \Delta(\lambda)\,\,\,\hbox{is not injective}\},\]
  the {\it continuous spectrum\/} $\sigma_c(\Delta)$ is defined by
\begin{equation}
\nonumber
\begin{split}
\sigma_c(\Delta) = \big\{\lambda\in {\mathbb C}:\, \Delta(\lambda)\,\,\hbox{is injective}\,\,{\mathscr R}(\Delta(\lambda))\not= V^*\,\,\hbox{and}\,\,\overline{{\mathscr R}(\Delta(\lambda))}=V^*\big\},
\end{split}
\end{equation}
and the  {\it  residual spectrum\/} $\sigma_r(\Delta)$ is defined by
\[
\sigma_r(\Delta) = \big\{\lambda\in {\mathbb C}:\, \Delta(\lambda)\,\,\hbox{is injective and}\,\,\overline{{\mathscr R}(\Delta(\lambda))}\not= V^*\big\}.\]
\end{definition}
\noindent Note that $\lambda\in \sigma_p(\Delta)$ if and only if there exists a nonzero $x\in V$ such that $\Delta(\lambda)x=0$. The value $\lambda\in \sigma_p(\Delta)$ is often called a {\it characteristic value\/} of $\Delta$.

Let $\sigma_p(A),\,\sigma_c(A)$ and $\sigma_r(A)$ and $\sigma_p({\cal A}),\,\sigma_c({\cal A})$ and $\sigma_r({\cal A})$ denote the point, continuous and residual spectrum sets of $A$ and ${\cal A}$, respectively. In connection with (\ref{12/09/16(1)}) and (\ref{23/08/13(212)}), we further define
\begin{equation}
\label{01/09/13(10)}
\begin{cases}
\Gamma_c = \{\lambda\in {\mathbb C}:\, n(\lambda)\not = 0,\,m(\lambda) n(\lambda)^{-1}\in \sigma_c(A)\},\\
\Gamma_r = \{\lambda\in {\mathbb C}:\, n(\lambda)\not = 0,\,m(\lambda) n(\lambda)^{-1}\in \sigma_r(A)\},\\
\Gamma_p = \{\lambda\in {\mathbb C}:\, n(\lambda)\not = 0,\,m(\lambda) n(\lambda)^{-1}\in \sigma_p(A)\},\\
\Gamma_0 = \{\lambda\in {\mathbb C}:\, \lambda\not = 0,\,n(\lambda)=0\},\\
\Gamma_1 = \{\lambda\in {\mathbb C}:\, n(\lambda)\not=0,\,m(\lambda) n(\lambda)^{-1}\in \sigma(A)\}.\\
\end{cases}
\end{equation}
By using Propositions 3.1--3.5, one can obtain the following result whose proof is similar to that one of Theorem 3.9 in \cite{dbgkkes85}.
\begin{proposition}
\label{22/08/13(2)}
 For the characteristic operator $\Delta(\lambda)$ and the associated generator ${\cal A}$ of the equation (\ref{11/08/2013(303)}), we have

(i) $\Gamma_0\subset \sigma_c({\cal A}) \subset \sigma_c(\Delta) =\Gamma_c\cup \Gamma_0;$

(ii) $\sigma_r({\cal A}) = \sigma_r(\Delta) = \Gamma_r;$

(iii)
\[
\sigma_p({\cal A}) = \sigma_p(\Delta) =
\begin{cases}
\Gamma_p\hskip 20pt &\hbox{if}\hskip 20pt 1 -\alpha_1 + \alpha_2 + \displaystyle\int^0_{-r} (\beta(\theta)- \gamma(\theta))d\theta\not= 0,\\
\Gamma_p\cup \{0\}\hskip 20pt &\hbox{if}\hskip 20pt 1 -\alpha_1 + \alpha_2 + \displaystyle\int^0_{-r} (\beta(\theta)- \gamma(\theta))d\theta= 0.
\end{cases}
\]
\end{proposition}

 If $\alpha_1=0$, $\gamma(\cdot)\equiv 0$ and $\beta(\cdot)\equiv 0$, then the equation (\ref{11/08/2013(303)}) reduces to a simple form
\begin{equation}
\label{11/08/2013(30378)}
\begin{cases}
dy(t) =Ay(t)dt + \alpha_2 A y(t-r)dt,\,\,\,\,t\ge 0,\\
y(0)=\phi_0,\,\,\,\,y_0=\phi_1,\,\,\,\phi=(\phi_0, \phi_1)\in {\cal H}.
\end{cases}
\end{equation}
Let us suppose at present that $A$ is some linear operator, e.g., Laplace operator, in conjunction with the form $a(\cdot, \cdot)$ in (\ref{29/01/2013(390)}) to generate a compact semigroup. It was shown, however,  by Di Blasio, Kunisch and Sinestrari \cite{dbgkkes85} that the associated $C_0$-semigroup $e^{t{\cal A}}$, $t\ge 0$, is generally not compact or even not eventually norm continuous, as shown by Jeong \cite{jj91}. A direct consequence of this fact is that we cannot use the well-known spectral mapping theorem to establish stability, based on the  spectrum knowledge of ${\cal A}$ for system (\ref{11/08/2013(30378)}).

 Bearing the above statement in mind, let us consider the following version of equation (\ref{11/08/2013(303)}) with distributed delay by taking $\alpha_1=\alpha_2=0$,
\begin{equation}
\label{11/08/2013(30070)}
\begin{cases}
d\Big(y(t)  - \displaystyle\int^0_{-r}\gamma(\theta)y(t+\theta)d\theta\Big)  =A\Big(y(t) - \displaystyle\int^0_{-r}\gamma(\theta)y(t+\theta)d\theta\Big)dt\\
\hskip 170pt + \displaystyle\int^0_{-r} \beta(\theta)A y(t+\theta)d\theta,\,\,\,\,t\ge 0,\\
y(0)=\phi_0,\,\,\,\,y_0=\phi_1,\,\,\,\phi=(\phi_0, \phi_1)\in {\cal H}.
\end{cases}
\end{equation}
It was shown by Liu \cite{kl2016(2)}
 that when the weight functions $\gamma(\cdot)$, $\beta(\cdot)$ satisfy
 \begin{equation}
 \label{06/11/16(3)}
 \gamma(\cdot)\in L^2([-r, 0], {\mathbb C}),\hskip 20pt \beta(\cdot)\in L^2([-r, 0], {\mathbb C}),
 \end{equation}
   the associated solution semigroup $e^{t{\cal A}}$, $t\ge 0$, in (\ref{11/08/2013(303)}) is eventually norm continuous for $t>r$, i.e, $e^{\cdot{\cal A}}:\, [0, \infty)\to {\mathscr L}({\cal H})$ is continuous on $(r, \infty)$, which implies further that the spectral mapping theorem  is fulfilled,
    \[
    \sup\{\hbox{Re}\,\lambda:\, \lambda\in \sigma({\cal A})\}=\inf\{\mu\in{\mathbb R}:\, \|e^{t{\cal A}}\|\le Me^{\mu t}\,\,\hbox{for some}\,\,M>0\}.
    \]

 Now let us consider the following stochastic functional differential equation of neutral type with additive noise,
    \begin{equation}
\label{11/08/2013(30070804)}
\begin{cases}
d\Big(y(t)  - \displaystyle\int^0_{-r}\gamma(\theta)y(t+\theta)d\theta\Big)  =A\Big(y(t) - \displaystyle\int^0_{-r}\gamma(\theta)y(t+\theta)d\theta\Big)dt\\
\hskip 170pt + \displaystyle\int^0_{-r} \beta(\theta)A y(t+\theta)d\theta + bdw(t),\,\,\,\,t\ge 0,\\
y(0)=\phi_0,\,\,\,\,y_0=\phi_1,\,\,\,\phi=(\phi_0, \phi_1)\in {\cal H},
\end{cases}
\end{equation}
 where $b\in H$ and $w(\cdot)$ is a standard real-valued Brownian motion.
 We can re-write (\ref{11/08/2013(30070804)}) as a stochastic differential equation without time delay in ${\cal H}$,
 \begin{equation}
 \label{07/03/2017(1)}
\begin{cases}
dY(t) ={\cal A}Y(t)dt + {\cal B}dw(t), \,\,\,\,t\ge 0,\\
Y(0)=\phi\in {\cal H},
\end{cases}
\end{equation}
where ${\cal A}$ is the generator given in Theorem \ref{06/03/2017(1)}, ${\cal B}=(b, 0)\in {\cal H}$
and $Y(t)=(y(t), y_t)$ for all $t\ge 0$. For equation (\ref{07/03/2017(1)}), if we can find conditions  by showing
\begin{equation}
\label{22/08/13(20)}
\sup\{Re\,\lambda:\, \lambda\in \sigma({\cal A})\}<0,
\end{equation}
to secure an exponentially stable semigroup $e^{t{\cal A}}$, $t\ge 0$, then we will obtain a unique stationary solution to the equation (\ref{11/08/2013(30070804)})  (cf, e.g., Pr\'ev\^ot and R\"ockner \cite{cpmr2007}).

\begin{proposition}
\label{27/08/13(40)}
 Suppose that the spectrum set $\sigma(A)\subset (-\infty, -c_0]$ for some $c_0>0$ and the functions $\gamma,\,\beta$ in (\ref{11/08/2013(30070804)}) satisfy
 \begin{equation}
 \label{23/03/17(1)}
 \|\gamma\|_{L^1([-r, 0], {\mathbb R})} +  \|\beta\|_{L^1([-r, 0], {\mathbb R})}<1.
 \end{equation}
Then there exists a unique stationary solution for the equation   (\ref{11/08/2013(30070804)}).
\end{proposition}
\begin{proof}  Note that from Proposition \ref{22/08/13(2)} we have $\sigma({\cal A}) \subset \Gamma_0\cup \Gamma_1.$
We shall show that under the assumptions in this proposition, there is a constant  $\mu>0$ such that Re$\,\lambda\le -\mu$ for all $\lambda\in \Gamma_0\cup \Gamma_1$ and hence for all $\lambda\in \sigma({\cal A})$.

First, for elements in  $\Gamma_0$, if there exist a sequence $\{\lambda_n\}\subset {\mathbb C}$ such that Re$\,\lambda_n\ge 0$ or Re$\,\lambda_n\to 0$ as $n\to\infty$, then by using (\ref{12/09/16(1)}),  (\ref{23/08/13(212)}) and Dominated Convergence Theorem, we have
\[
\begin{split}
1 &\le \limsup_{n\to\infty}\Big|\int^0_{-r} (\gamma(\theta) + \beta(\theta))e^{\lambda_n\theta}d\theta\Big| \le \int^0_{-r} (|\gamma(\theta)| + |\beta(\theta)|)d\theta<1,
\end{split}
\]
which is clearly a contradiction. Thus the desired result is obtained.

Now we consider the elements of $\sigma({\cal A})$ in $\Gamma_1$.
If there exists a $\lambda$ of $\sigma({\cal A})$ in $\Gamma_1$ such that Re$\,\lambda\ge 0$ (the case that Re$\,\lambda\to 0$ can be similarly proved) with $m(\lambda)/n(\lambda) =: -\delta \le -c_0$, then we get by taking the real part of the equation into account that
\[
\begin{split}
&1 + \frac{Re\, m(\lambda)}{\delta}\\
 &= 1 + \frac{Re\,\lambda -\displaystyle\int^0_{-r} \gamma(\theta)e^{(Re\lambda)\theta}\cos((Im\,\lambda)\theta)d\theta - (Im\,\lambda)\displaystyle\int^0_{-r}\gamma(\theta)e^{(Re\lambda)\theta}\sin((Im\,\lambda)\theta)d\theta}{\delta}\\
& = -\int^0_{-r} \beta(\theta)e^{(Re\,\lambda)\theta}\cos [(Im\,\lambda)\theta]d\theta + \int^0_{-r} \gamma(\theta)e^{(Re\,\lambda)\theta}\cos[(Im\,\lambda)\theta]d\theta.
\end{split}
\]
Since $|Im\,\lambda|\le M$ for some constant $M>0$ (see Th. 2.22, p.35, \cite{absp05}), we get immediately that
\[
\begin{split}
1 \le  1 + \frac{Re\,m(\lambda)}{c_0}
\le&\,\, \Big|\int^0_{-r} (\gamma(\theta) - \beta(\theta))e^{(Re\,\lambda)\theta}\cos [(Im\,\lambda)\theta]d\theta\Big|\\
=&\,\, \int^0_{-r} (|\gamma(\theta)| + |\beta(\theta)|)d\theta <1,
\end{split}
\]
which, once again, yields a contradiction. Combining the above results, we thus obtain that
\[
Re\,\lambda\le -\mu\hskip 10pt \hbox{for some}\hskip 10pt \mu>0\hskip 10pt \hbox{and all}\hskip 10pt \lambda\in\sigma({\cal A}).\]
Therefore, the solution semigroup $e^{t{\cal A}}$, $t\ge 0$, is exponentially stable.
This fact further implies the existence of a unique stationary solution to  (\ref{11/08/2013(30070)}). The proof is  complete.
\end{proof}

\begin{example}\rm
Consider the following initial-boundary value problem of Dirichlet type of the stochastic neutral Laplace equation:
\begin{equation}
\label{22/08/13(7090)}
\begin{cases}
d\Big(y(t, x) - \displaystyle\int^0_{-r}\kappa e^{\mu\theta} y(t+\theta, x)d\theta\Big) = \displaystyle\frac{\partial^2}{\partial x^2}\Big(y(t, x) - \int^0_{-r}\kappa e^{\mu\theta} y(t+\theta, x)d\theta\Big)dt\\
  \hskip 100pt +\displaystyle\int^0_{-r}\alpha\frac{\partial^2 y(t+\theta, x)}{\partial x^2}d\theta dt + b(x)dw(t),\,\,\,\,t\ge 0,\,\,\,\,x\in {\cal O},\\
y(0, \cdot) =\phi_0(\cdot)\in L^2({\cal O}; {\mathbb R}),\\
y(t, \cdot)=\phi_1(t, \cdot)\in H_0^{1}({\cal O}; {\mathbb R}),\,\,\,\,\hbox{a.e.}\,\,\,\,t\in [-r, 0).
\end{cases}
\end{equation}
Here ${\cal O}$ is a bounded open subset of ${\mathbb R}^n$ with regular boundary $\partial {\cal O}$, $\alpha,\,\kappa,\,\mu\in {\mathbb R}$, $r>0$ and $b(\cdot)\in L^2({\cal O}; {\mathbb R})$.

We can re-write (\ref{22/08/13(7090)}) as a stochastic neutral initial boundary problem  (\ref{11/08/2013(30070)}) in the Hilbert space $H= L^2({\cal O}; {\mathbb R})$ by setting
\[
\begin{cases}
A = \displaystyle\frac{\partial^2}{\partial x^2},\\
V= H^{1}_0({\cal O}; {\mathbb R}),\\
\gamma(\theta) = \kappa e^{\mu\theta},\,\,\,\beta(\theta) \equiv  \alpha,\,\,\,\theta\in [-r, 0].
\end{cases}
\]
Then for any random initial $(\phi_0, \phi_1) \in {\cal H}$, there exists a unique strong solution to (\ref{22/08/13(7090)}) defined in $[0, \infty)$. Furthermore, by applying the results derived in this section to (\ref{22/08/13(7090)}), one can obtain a unique stationary solution. In fact, note that $A= \partial^2/\partial x^2$ is a self-adjoint and negative operator in $H$ and its spectrum satisfies $\sigma(A) = \sigma_p(A) \subset (-\infty, -c_0]$ for some $c_0>0$. Then by Proposition \ref{27/08/13(40)} and a direct computation, we obtain that if
\[
|\alpha|< \frac{1}{r}\hskip 20pt\hbox{and}\hskip 20pt |\kappa| \le
\begin{cases}
 e^{r\mu}(1-|\alpha|r)/r\hskip 30pt &\hbox{if}\hskip 15pt \mu\le 0,\\
(1-|\alpha|r)/r\hskip 30pt &\hbox{if}\hskip 15pt  \mu> 0,
\end{cases}
\]
 the associated solution semigroup of  (\ref{22/08/13(7090)}) is exponentially stable. Moreover, the lift-up system  (\ref{07/03/2017(1)}) of equation (\ref{22/08/13(7090)}) in this case has a unique stationary solution.
\end{example}

\end{document}